\documentclass[12pt]{amsart}

\usepackage{mathptmx} 
\DeclareSymbolFont{cmletters}{OML}{cmm}{m}{it}              
\DeclareSymbolFont{cmsymbols}{OMS}{cmsy}{m}{n}
\DeclareSymbolFont{cmlargesymbols}{OMX}{cmex}{m}{n}
\DeclareMathSymbol{\myjmath}{\mathord}{cmletters}{"7C}     \let\jmath\myjmath 
\DeclareMathSymbol{\myamalg}{\mathbin}{cmsymbols}{"71}     \let\amalg\myamalg
\DeclareMathSymbol{\mycoprod}{\mathop}{cmlargesymbols}{"60}
\DeclareMathSymbol{\myalpha}{\mathord}{cmletters}{"0B}     \let\alpha\myalpha 
\DeclareMathSymbol{\mybeta}{\mathord}{cmletters}{"0C}      \let\beta\mybeta
\DeclareMathSymbol{\mygamma}{\mathord}{cmletters}{"0D}     \let\gamma\mygamma
\DeclareMathSymbol{\mydelta}{\mathord}{cmletters}{"0E}     \let\delta\mydelta
\DeclareMathSymbol{\myepsilon}{\mathord}{cmletters}{"0F}   \let\epsilon\myepsilon
\DeclareMathSymbol{\myzeta}{\mathord}{cmletters}{"10}      \let\zeta\myzeta
\DeclareMathSymbol{\myeta}{\mathord}{cmletters}{"11}       \let\eta\myeta
\DeclareMathSymbol{\mytheta}{\mathord}{cmletters}{"12}     \let\theta\mytheta
\DeclareMathSymbol{\myiota}{\mathord}{cmletters}{"13}      \let\iota\myiota
\DeclareMathSymbol{\mykappa}{\mathord}{cmletters}{"14}     \let\kappa\mykappa
\DeclareMathSymbol{\mylambda}{\mathord}{cmletters}{"15}    \let\lambda\mylambda
\DeclareMathSymbol{\mymu}{\mathord}{cmletters}{"16}        \let\mu\mymu
\DeclareMathSymbol{\mynu}{\mathord}{cmletters}{"17}        \let\nu\mynu
\DeclareMathSymbol{\myxi}{\mathord}{cmletters}{"18}        \let\xi\myxi
\DeclareMathSymbol{\mypi}{\mathord}{cmletters}{"19}        \let\pi\mypi
\DeclareMathSymbol{\myrho}{\mathord}{cmletters}{"1A}       \let\rho\myrho
\DeclareMathSymbol{\mysigma}{\mathord}{cmletters}{"1B}     \let\sigma\mysigma
\DeclareMathSymbol{\mytau}{\mathord}{cmletters}{"1C}       \let\tau\mytau
\DeclareMathSymbol{\myupsilon}{\mathord}{cmletters}{"1D}   \let\upsilon\myupsilon
\DeclareMathSymbol{\myphi}{\mathord}{cmletters}{"1E}       \let\phi\myphi
\DeclareMathSymbol{\mychi}{\mathord}{cmletters}{"1F}       \let\chi\mychi
\DeclareMathSymbol{\mypsi}{\mathord}{cmletters}{"20}       \let\psi\mypsi
\DeclareMathSymbol{\myomega}{\mathord}{cmletters}{"21}     \let\omega\myomega
\DeclareMathSymbol{\myvarepsilon}{\mathord}{cmletters}{"22}\let\varepsilon\myvarepsilon
\DeclareMathSymbol{\myvartheta}{\mathord}{cmletters}{"23}  \let\vartheta\myvartheta
\DeclareMathSymbol{\myvarpi}{\mathord}{cmletters}{"24}     \let\varpi\myvarpi
\DeclareMathSymbol{\myvarrho}{\mathord}{cmletters}{"25}    \let\varrho\myvarrho
\DeclareMathSymbol{\myvarsigma}{\mathord}{cmletters}{"26}  \let\varsigma\myvarsigma
\DeclareMathSymbol{\myvarphi}{\mathord}{cmletters}{"27}    \let\varphi\myvarphi

\usepackage{amsthm,amsmath,amssymb,amscd,graphics,enumerate,latexsym,verbatim,stmaryrd,graphicx,color}
\usepackage[all]{xy}

\theoremstyle{plain}
\newtheorem{thm}{Theorem}[section]

\newtheorem{prop}[thm]{Proposition}

\theoremstyle{definition}
\newtheorem{df}[thm]{Definition}

\newtheorem{rem}[thm]{Remark}

\DeclareMathOperator{\Gal}{Gal}

\DeclareMathOperator{\Spec}{Spec}

\DeclareMathOperator{\Div}{Div}
\DeclareMathOperator{\Cl}{Cl}
\DeclareMathOperator{\Pic}{Pic}

\def\0{{\bf 0}}
\def\A{{\mathbb A}}

\def\C{{\mathbb C}}
\def\F{{\mathbb F}}

\def\N{{\mathbb N}}
\def\P{{\mathbb P}}
\def\Q{{\mathbb Q}}
\def\R{{\mathbb R}}
\def\Z{{\mathbb Z}}

\def\cI{{\mathcal I}}

\def\cL{{\mathcal L}}
\def\cM{{\mathcal M}}

\def\cO{{\mathcal O}}
\def\tcO{{\widetilde{\mathcal O}}}
\def\cP{{\mathcal P}}

\def\cR{{\mathcal R}}

\def\cT{{\mathcal T}}

\def\fm{{\mathfrak m}}
\def\fp{{\mathfrak p}}

\def\Fun{{\F_1}}
\def\Funsq{{\F_{1^2}}}
\def\Funn{{\F_{1^n}}}
\def\Funinf{{\F_{1^\infty}}}

\def\1{\textbf{1}}

\def\overz{\overline{\Spec\Z}}
\def\overX{\overline{X}}

\def\blanc{-}
\def\bp{{\mathcal{B}lpr}}
\def\bpspaces{{\mathcal{L}\!oc\bp \mathcal{S}\!p}}

\def\Rings{{\mathcal{R}\textit{ings}}}

\def\={\equiv}
\def\n={\equiv\hspace{-10,5pt}/\hspace{3,5pt}}

\def\top{{\textup{top}}}
\def\arak{{\textup{}}}
\def\eff{{\textup{eff}}}

\def\toptimes{\times^\top}

\DeclareMathOperator{\BSch}{Sch_\Fun}

\newdir{ >}{{}*!/-5pt/@^{(}}\newcommand{\arincl}[1]{\ar@{ >->}@<-0,0ex>#1} 

\newcommand{\norm}[1]{\left| #1 \right|}

\newcommand{\gen}[1]{\langle #1 \rangle}

\newcommand{\bpquot}[2]{#1\!\sslash\!#2}
\newcommand{\bpgenquot}[2]{#1\!\sslash\!\gen{#2}}

\def\Re{\textup{Re}}

\begin{document}

\title{Blueprints - towards absolute arithmetic?}
\author{Oliver Lorscheid}
\date{}
\address{Instituto Nacional de Matem\'atica Pura e Aplicada, Estrada Dona Castorina 110, 22460-320 Rio de Janeiro, Brazil.}
\email{lorschei@impa.br}

\begin{abstract}
 One of the driving motivations to develop $\Fun$-geometry is the hope to translate Weil's proof of the Riemann hypothesis from positive characteristics to number fields, which might result in a proof of the classical Riemann hypothesis. The underlying idea is that the spectrum of $\Z$ should find an interpretation as a curve over $\Fun$, which has a completion $\overz$ analogous to a curve over a finite field. The hope is that intersection theory for divisors on the arithmetic surface $\overz\times\overz$ will allow to mimic Weil's proof.

 It turns out that it is possible to define an object $\overz$ from the viewpoint of blueprints that has certain properties, which come close to the properties of its analogs in positive characteristic. This shall be explained in the following note, which is a summary of a talk given at the Max Planck Institute in March, 2012. 
\end{abstract}

\maketitle


\section*{Introduction}
\label{intro}

\noindent
For a not yet systematically understood reason, many arithmetic laws have (conjectural) analogues for function fields and number fields. While in the function field case, these laws often have a conceptual explanation by means of a geometric interpretation, methods from algebraic geometry break down in the number field case. The mathematical area of $\Fun$-geometry can be understood as a program to develop a geometric language that allows to transfer the geometric methods from function fields to number fields.

A central problem of this kind, which lacks a proof in the number field case, is the Riemann hypothesis. The Riemann zeta function is defined as the Riemann sum $\zeta(s)=\sum_{n\geq1}n^{-1}$, or, equivalently, as the Euler product $\prod_{p\text{ prime}}(1-p^{-s})^{-1}$ (these expressions converge for $\Re(s)>1$, but can be continued to meromorphic function on $\C$). A more symmetric expression with respect to its functional equation is given by the \emph{completed zeta function}
\[
 \zeta^\ast(s) \quad = \quad \underbrace{\pi^{-s/2}\Gamma(\frac s2)}_{\zeta_\infty(s)} \ \cdot \ \prod_{p\text{ prime}} \ \underbrace{\frac{1}{1-p^{-s}}}_{\zeta_p(s)} \hspace{2cm} (\text{for }\Re(s)>1)
\]
where we call the factor $\zeta_p(s)$ the \emph{local zeta factor at $p$} for $p\leq\infty$. The meromorphic continuation of $\zeta^\ast(s)$ to $\C$ satisfies $\zeta^\ast(s)=\zeta^\ast(1-s)$. The fundamental conjecture is the

\medskip\noindent
\textbf{Riemann hypothesis:} If $\zeta^\ast(s)=0$, then $\Re(s)=1/2$.
\smallskip

The analogueous statement for the function field of a curve $X$ over a finite fields has been proven by Andr\'e Weil more than seventy years ago. Weyl's proof uses intersection theory for the self-product $X\times X$ and the Lefschetz fix-point formula for the absolute Frobenius action on $X$ resp.\ $X\times X$. There were several attempts to translate the geometric methods of this proofs into arithmetic arguments that would apply for number fields as well, but only with partial success so far.

A different approach, primarily due to Grothendieck, is the development of a theory of (mixed) motives, but such a theory relies on the solution of some fundamental problems. In particular, Deninger formulates in \cite{Deninger92} a conjectural formalism for a \emph{big site $\cT$} of motives that should contain a compactification $\overz=\Spec\Z\cup\{\infty\}$ of the arithmetic line, and he conjectured that the formula
\[
 \zeta^\ast(s) \quad = \quad \frac{\det_\infty\Bigl(\frac 1{2\pi}(s-\Theta)\Bigl| H^1(\overline{\Spec\Z},\cO_\cT)\Bigr.\Bigr)}{\det_\infty\Bigl(\frac 1{2\pi}(s-\Theta)\Bigl| H^0(\overline{\Spec\Z},\cO_\cT)\Bigr.\Bigr) \ \cdot \ \det_\infty\Bigl(\frac 1{2\pi}(s-\Theta)\Bigl| H^2(\overline{\Spec\Z},\cO_\cT)\Bigr.\Bigr)} 
\]
holds true where $\det_\infty$ denotes the regularized determinant, $\Theta$ is an endofunctor on $\cT$ and $H^i(\blanc,\cO_\cT)$ is a certain proposed cohomology. This description combines with Kurokawa's work on multiple zeta functions (\cite{Kurokawa92}) to the hope that there are motives $h^0$ (\emph{``the absolute point''}), $h^1$ and $h^2$ (\emph{``the absolute Tate motive''}) with zeta functions
$$ \zeta_{h^w}(s) \ = \ \det{}_\infty\Bigl(\frac 1{2\pi}(s-\Theta)\Bigl| H^w(\overline{\Spec\Z},\cO_\cT)\Bigr.\Bigr) $$

The viewpoint of $\Fun$-geometry is that $\overz$ should be a curve over the elusive field $\Fun$ with one element\footnote{The origins of this idea are not completely clear to me. The first mentioning in literature is in the lecture notes \cite{Manin95} of Yuri Manin (the lectures were held in 1991 and 1992). Alexander Smirnov told me that he that this idea occured to him already in 1985.}. A good theory of geometry over $\Fun$ might eventually allow to transfer Weil's original proof to the number field case. Manin tied up this viewpoint with the motivic approach: he proposed in \cite{Manin95} to interpret the absolute point $h^0$ as $\Spec\Fun$ and the absolute Tate motive $h^2$ as the affine line over $\Fun$.

Soul\'e made the first attempt to define a class of objects that should be thought of as varieties over $\Fun$ and he defines the zeta function of a variety over $\Fun$ (cf.\ \cite{Soule04}). Indeed, his candidates for the absolute point $\Spec\Fun$ has the zeta function $\zeta_{\Spec\Fun}(s)=s$ and the zeta function of the affine line $\A^1_\Fun$ over $\Fun$ has the zeta function $\zeta_{\Spec\Fun}(s)=s-1$, which fits Deninger's formalism up to a factor $2\pi$. 

However, the cohomological interpretation $\det_\infty\bigl((s-\Theta)/(2\pi)\bigl| H^i(\overline{\Spec\Z},\cO_\cT)\big.\bigr)$ of the motivic zeta function is still mysterious. The first step is to make sense of the object $\overz$, and, indeed, there were several attempts to do so, for instance by Durov (cf.\ \cite{Durov07}) and Haran (cf.\ \cite{Haran07}). To transfer Weil's proof of the Riemann hypothesis for function fields one needs to consider intersection theory on the arithmetic surface $\overz\times_\Fun\overz$. But in the approaches towards $\Fun$-geometry so far\footnote{During the time of writing the preprint \cite{Takagi12} appeared, in which Takagi introduces a new candidate for $\overz$. This approach seems to be close to the one explained in this note, but I do not yet understand the connection clearly.}---even if one considers only the affine part $\Spec\Z$ of $\overz$---the self-product $\Spec\Z\times_\Fun\Spec\Z$ is either isomorphic to $\Spec\Z$ itself or infinite-dimensional over $\Z$.

Aim of this note is to present a definition of $\overz$ as a blueprinted space (see \cite{blueprints1} and \cite{blueprints2}), which is close to Arakelov geometry and whose self-product over $\Fun$ (or better over $\Funsq$, as explained in the following) is $2$-dimensional.

\section{The naive definition of $\overz$}
\label{section: naive definition of spec z}

\noindent
In analogy to curves over finite fields, one might expect the compactification of $\Spec\Z$ to be a sober topological space $X$ that consists of a unique generic point $\eta$ and a closed point $p$ for every (non-trivial) place $\norm\ _p$ of the ``function field'' $\Q$ of $X=\overz$. In other words, a closed point is either a (finite) prime $p<\infty$ or the archimedean place $p=\infty$, which we will also call the infinite prime. The closed sets of $X$ are the finite sets of places $\{p_1,\dotsc,p_n\}$ and $X$. Further, there should be a structure sheaf $\cO_X$, which associates to the open set $U=X-\{p_1,\dotsc,p_n\}$ the following set of rational functions:
\[
 \cO_X \bigl( U \bigr) \quad = \quad \left\{ \ \frac ab \in\Q \ \left| \ \norm{\frac ab}_q\leq 1\text{ for all }q\notin\{p_1,\dotsc,p_n\}\ \right.\right\}.
\]
As a consequence, the global sections are $\Gamma(X,\cO_X)=\cO_X(X)=\{0,\pm 1\}$, which should be thought of as the constants of $\overz$. The stalks of $\cO_X$ are 
\[
 \cO_{X,p} \quad = \quad \left\{ \ \frac ab \in\Q \ \left| \ \norm{\frac ab}_p\leq 1\ \right.\right\},
\]
with ``maximal ideals''
\[
 \fm_p \quad = \quad \left\{ \ \frac ab \in\Q \ \left| \ \norm{\frac ab}_p< 1\ \right.\right\}
\]
for every prime $p\leq\infty$. One observes that $X$ is indeed an extension of the scheme $\Spec\Z$, i.e.\ the restriction of $X$ to $U=X-\{\infty\}$ can be identified with $\Spec\Z$. The problem with this definition is that the sets $\cO_X\bigl( X-\{p_1,\dotsc,p_n\} \bigr)$ are not sub\emph{rings} of $\Q$ if $\infty\notin\{p_1,\dotsc,p_n\}$, and neither is the stalk $\cO_{X,\infty}=[-1,1]\cap\Q$ at infinity. 

However, all of these sets are \emph{monoids}, by which we mean in this note (multiplicatively written) commutative semigroups with identity $1$ and absorbing element $0$. One of the tasks of an $\Fun$-geometry is to give this naive definition of $\overz$ a precise meaning, i.e.\ making it a geometric object in a meaningful category. We will present one possible approach via blueprints in the following.

\section{Blueprints}
\label{section: blueprints}

\noindent
 The definition of a blueprint was initially motivated by Jacques Tits' idea of descending Chevalley groups to $\Fun$: there was a need to go beyond Deitmar's $\Fun$-geometry attached to monoids by including a relation on the formal sums in the elements of a monoid. 

 Meanwhile it became clear that blueprints not only yield a satisfactory answer to Tits' problem, with applications to (idempotent and tropical) semirings (cf.\ \cite{blueprints2}), but provide the natural background to understand certain aspects of geometry: buildings and their apartments; canonical bases and total positivity; cluster algebras; moduli of quiver representations and quiver Grassmannians. In this note, we restrict our attention to the definition of $\Spec\Z$ as a blueprinted space.

 We recall the definition of a blueprint. We will follow the conventions of \cite{blueprints2}, i.e.\ we mean by a blueprint what was called a \emph{proper blueprint with $0$} in \cite{blueprints1}. 

\begin{df}
 A \emph{blueprint $B$} is a monoid $A$ together with a \emph{pre-addition} $\cR$, i.e.\ $\cR$ is an equivalence relation on the semiring $\N[A]=\{\sum a_i|a_i\in A\}$ of finite formal sums of elements of $A$ that satisfies the following axioms (where we write $\sum a_i\=\sum b_j$ whenever $(\sum a_i,\sum b_j)\in\cR$):
\begin{enumerate}
 \item\label{ax1} If $\sum a_i\=\sum b_j$ and $\sum c_k\=\sum d_l$, then $\sum a_i+\sum c_k\=\sum b_j+\sum d_l$ and $\sum a_ic_k\=\sum b_jd_l$.
 \item\label{ax2} The absorbing element $0$ of $A$ is identified with the zero of $\N[A]$, i.e.\ $0\=(\text{empty sum})$.
 \item\label{ax3} If $a\= b$, then $a=b$ (as elements in $A$).
\end{enumerate}
 A \emph{morphism $f:B_1\to B_2$ between blueprints} is a multiplicative map $f:A_1\to A_2$ between the underlying monoids of $B_1$ and $B_2$, respectively, with $f(0)=0$ and $f(1)=1$ such that for every relation $\sum a_i\=\sum b_j$ in the pre-addition $\cR_1$ of $B_1$, the pre-addition $\cR_2$ of $B_2$ contains the relation $\sum f(a_i)\=\sum f(b_j)$. Let $\bp$ be the category of blueprints.
\end{df}

In the following, we write $B=\bpquot A\cR$ for a blueprint $B$ with underlying monoid $A$ and pre-addition $\cR$. We adopt the conventions used for rings: we identify $B$ with the underlying monoid $A$ and write $a\in B$ or $S\subset B$ when we mean $a\in A$ or $S\subset A$, respectively. Further, we think of a relation $\sum a_i\=\sum b_j$ as an equality that holds in $B$ (without the elements $\sum a_i$ and $\sum b_j$ being defined, in general). Given a set $S$ of relations, there is a smallest equivalence relation $\cR$ on $\N[A]$ that contains $S$ and satisfies \eqref{ax1} and \eqref{ax2}. If $\cR$ satisfies also \eqref{ax3}, then we say that $\cR$ is the pre-addition generated by $S$, and we write $\cR=\gen S$. In particular, every monoid $A$ has a smallest pre-addition $\cR=\gen\emptyset$.

\subsection{Relation to rings}\label{subsection: ralation to rings}
The idea behind the definition of a blueprint is that it is a blueprint of a ring (in the literal sense): given a blueprint $B=\bpquot A\cR$, one can construct the ring $B_\Z^+=\Z[A]/I(\cR)$ where $I(\cR)$ is the ideal $\{\sum a_i-\sum b_j\in \Z[A]|\sum a_i\=\sum b_j \text{ in }\cR\}$. This association is functorial by extending a blueprint morphism $f:B_1\to B_2$ linearly to a ring homomorphism $f_\Z^+:B_{1,\Z}^+\to B_{2,\Z}^+$, i.e. we obtain a functor $(\blanc)^+_\Z:\bp\to\Rings$ from blueprints to rings.

On the other hand, rings can be seen as blueprints: given a (commutative and unital) ring $R$, we can define the blueprint $B=\bpquot A\cR$ where $A$ is the underlying multiplicative monoid of $R$ and $\cR=\{\sum a_i\=\sum b_j|\sum a_i=\sum b_j\text{ in }R\}$. Under this identification, ring homomorphisms are nothing else than blueprint morphisms, i.e.\ we obtain a full embedding $\iota_\cR:\Rings\to\bp$. This allows us to identify rings with a certain kind of blueprints, and we can view blueprints as a generalization of rings. Accordingly, we call blueprints in the essential image of $\iota_\cR$ rings.

\subsection{Relation to monoids}\label{subsection: relation to monoids}
Another important class of blueprints are monoids. Namely, a monoid $A$ defines the blueprint $B=\bpgenquot A\emptyset$. This defines a full embedding $\iota_\cM:\cM\to\bp$ of the category $\cM$ of monoids into $\bp$. This justifies that we may call blueprints in the essential image of $\iota_\cM$ monoids and that we identify in this case $B$ and $A$.

\subsection{Cyclotomic field extensions and the archimedean valuation blueprint}\label{subsection: valuation blueprints}
We give some other examples, which are of interest for the purposes of this text. The initial object in $\bp$ is the monoid $\Fun:=\{0,1\}$, the so-called \emph{field with one element}. More general, we define the \emph{cyclotomic field extension $\Funn$ of $\Fun$} as the blueprint $B=\bpquot A\cR$ where $A=\{0\}\cup\mu_n$ is the union of $0$ with a cyclic group $\mu_n=\{\zeta_n^i|i=1,\dotsc,n\}$ of order $n$ with generator $\zeta_n$ and where $\cR$ is generated by the relations $\sum_{i=0}^{n/d}\zeta_n^{di}\=0$ for every divisor $d$ of $n$. The associated ring of $\Funn$ is the ring $\Z[\zeta_n]$ of integers of the cyclotomic field extension $\Q[\zeta_n]$ of $\Q$ that is generated by the $n$-th roots of $1$. Note that this breaks with the convention of $\Fun$-literature that $\Funn$ should be a monoid and its associated ring should be isomorphic to $\Z[T]/(T^n-1)$.

Of particular importance for us will be the blueprint $\Funsq=\bpgenquot{\{0,\pm1\}}{1+(-1)\=0}$, which we can identify with $\Gamma(X,\cO_X)$. Consequently, we can think of $X=\overz$ as a curve defined over $\Funsq$, which is an alteration of the prominent viewpoint in $\Fun$-geometry that $\overz$ should be a curve over $\Fun=\{0,1\}$. This will be of importance for the definition of cohomology (cf.\ Section \ref{section: cohomology}).

The stalk $\tcO_{X,\infty}=[-1,1]\cap\Q$ has a natural blueprint structure in terms of the pre-addition $\cR_\infty=\{\sum a_i\=\sum b_j|\sum a_i=\sum b_j\text{ in }\Q\}$. We will see, however, that this blueprint does not fit our purpose of making $\overz$ a well-behaved geometric object (cf.\ Section \ref{section: spec z as a locally bleuprinted space}).

\section{Locally blueprinted spaces}
\label{section: locally blueprinted spaces}

\noindent
Let $B=\bpquot A\cR$ be a blueprint. An \emph{ideal} of $B$ is a subset $I$ satisfying that $IB\subset I$ and that for every additive relation of the form $\sum a_i+ c\=\sum b_j$ in $B$ with $a_i,b_j\in I$, we have $c\in I$. The relation $0\=\text{(empty sum)}$ implies that every ideal $I$ contains $0$.

A \emph{multiplicative set} is a subset $S$ of $B$ that is closed under multiplication and contains $1$. An ideal $\fp$ of $B$ is a \emph{prime ideal} if its complement in $B$ is a multiplicative set. A \emph{maximal ideal} is an ideal that is maximal for the inclusion relation and not equal to $B$ itself. A blueprint is \emph{local} if it has a unique maximal ideal. A morphism $f:B_1\to B_2$ between local blueprints is \emph{local} if it maps the maximal ideal of $B_1$ to the maximal ideal of $B_2$.

Note that if $B$ is a ring, all the above definitions specialize to the corresponding definitions for rings. The same is true for monoids. The following well-known facts for rings generalize to blueprints: maximal ideals are prime ideals; if $B$ is a local blueprint, then its maximal ideal is the complement of the group of units, i.e. the set of elements with a multiplicative inverse.

A \emph{blueprinted space} is a topological space $X$ together with a sheaf $\cO_X$ in $\bp$. A morphism of blueprinted spaces is a continuous map together with a sheaf morphism. Since the category $\bp$ contains directed limits, the stalks $\cO_{X,x}$ in points $x\in X$ exist, and a morphism of blueprinted spaces induces morphisms between stalks. A \emph{locally blueprinted space} is a blueprinted space whose stalks $\cO_{X,x}$ are local blueprints with maximal ideal $\fm_x$ for all $x\in X$. A \emph{local morphism} between locally blueprinted spaces is a morphism of blueprinted spaces that induces local morphisms of blueprints between all stalks. We denote the resulting category by $\bpspaces$.

Note that we can define the residue field of a point $x$ of a locally blueprinted space $X$ as $\kappa(x)=\cO_{X,x}/\fm_x$. A local morphism of locally blueprinted spaces induces morphisms between residue fields.

The \emph{spectrum of a blueprint $B$} is defined analogously to the case of rings or monoids: $\Spec B$ is the locally blueprinted space whose underlying set $X$ is the set of all prime ideals of $B$, endowed with the Zariski topology, and whose structure sheaf $\cO_X$ consists of localizations of $B$. A \emph{blue scheme} is a locally blueprinted space that is locally isomorphic to spectra of blueprints. We denote the full subcategory of $\bpspaces$ whose objects are blue schemes by $\BSch$.

However, our definition of $\overz$ will not be a blue scheme, but merely a locally blueprinted space. An important effect in the category of blueprints is that colimits and tensor products of rings inside $\bp$ is, in general not a ring. This leads to a simpler structure of fibre products of blue schemes, which coincide with the fibre products in $\bpspaces$. More precisely, we have:

\begin{thm}\label{thm: fibre products}
 The category $\bpspaces$ has fibre products. The fibre product $X\times _S Y$ is naturally a subset of the topological product $X\toptimes Y$, and it carries the subspace topology. In the case of $S=\Spec\Funsq$, it has the explicit description
 \[
  X\times_\Funsq Y \qquad = \qquad \left\{ \ (x,y)\in X\toptimes Y \ \left| \ \substack{\text{there are a field }k\text{ and blueprint}\\ \text{morphisms } \kappa(x)\to k\text{ and }\kappa(y)\to k} \ \right.\right\}.
 \]
 If $X$, $Y$ and $S$ are blue schemes, then the fibre product $X\times_S Y$ in $\bpspaces$ coincides with the fibre product in $\BSch$. In particular, $X\times_S Y$ is a blue scheme.
\end{thm}

\section{$\overz$ as a locally blueprinted space}
\label{section: spec z as a locally bleuprinted space}

\noindent
 We investigate the spectrum of the blueprint $\tcO_{X,\infty}$ as defined in Section \ref{subsection: valuation blueprints}. The multiplicative structure of $\tcO_{X,\infty}=\bpquot{A_\infty}{\cR_\infty}$ for the monoid $A_\infty=[-1,1]\cap \Q$ and $\cR_\infty$ as in Section \ref{subsection: valuation blueprints} implies that if an ideal $I$ of $\tcO_{X,\infty}$ contains an element $a$, then it contains $[-a,a]\cap \Q=aA_\infty$. Consequently, if $I$ contains any non-zero element $a$, then it contains also elements $a_1,\dotsc,a_n$ such that $\sum a_i\=1$. The additive axiom for ideals implies that $1\in I$. Thus $(0)$ and $(1)$ are the only ideals of $\tcO_{\tilde X,\infty}$, and its spectrum consists of one single point $\eta=(0)$.

 In particular, the set $\fm_\infty=\{a\in\Q|\norm a_\infty<1\}$ is not an ideal, which prevents $\tcO_{X,\infty}$ to be a candidate for the stalk of $\overz$ at infinity. However, if we endow the monoid $[-1,1]\cap\Q$ with a weaker pre-addition, then we obtain a useful blueprint. Namely, the set of ideals of the blueprint $\cO_{X,x}=\bpgenquot{A_\infty}{1+(-1)\=0}$ is, by the above considerations,
 \[ 
  \cI_\infty \quad = \quad \{ \ (-a,a)\cap\Q \ | \ a\in(0,1] \ \} \quad \amalg \quad \{ \ [-a,a]\cap\Q \ | \ a\in[0,1]\cap\Q \ \},
 \]
 and the prime ideals of $\cO_{X,x}$ are $(0)$ and $\fm_x$, in analogy to the spectrum of a discrete valuation ring.

 This difference of the spectra of $\tcO_{X,\infty}$ and $\cO_{X,\infty}$ does not occur for finite primes $p$. Namely, the set of ideals of the blueprint $\cO_{X,p}=\bpgenquot{A_p}{1+(-1)\=0}$ is
 \[
  \cI_p \quad = \quad \{ \ (p^i) \ | \ i\in\N\cup\{\infty\}\}
 \]
 where $(p^0)=\cO_{X,p}$, $(p^1)=\fm_p$ and $(p^\infty)=(0)$. This set is identical with the set of ideals of the ring $\tcO_{X,p}=\Z_{(p)}$, which is the stalk of $\Spec\Z$ at $p$. The spectrum of $\cO_{X,p}$ consists of the two points $(0)$ and $\fm_p$ as desired. 

 We can define a sheaf $\cO_X$ on $X=\overz$ that yields $\cO_{X,p}$ as the stalk for any (finite or infinite) prime $p$. Namely, we put
 \[
  \cO_X\bigl( X-\{p_1,\dotsc,p_n\} \bigr) \quad = \quad \bpgenquot{\Bigl\{ \ \frac ab \in\Q \ \Bigl| \ \norm{\frac ab}_q\leq 1\text{ for all }q\notin\{p_1,\dotsc,p_n\}\ \Bigr\}}{1+(-1)\=0}.
 \]
 This makes $X$ a locally blueprinted space over $\Funsq$, i.e. $X$ comes together with a morphism $X\to\Spec\Funsq$. The underlying topological space of $X$ is Noetherian and has dimension $1$. Its structure sheaf is coherent in the sense that $X$ has an open covering $\{U_i\}$ such that for every $V\subset U_i$, the restriction map $\cO_X(U_i)\to\cO_X(V)$ is a localization. 

 Since for every prime $p$, the residue field $\kappa(p)$ can embedded into fields of any characteristic, Theorem \ref{thm: fibre products} implies the following.

 \begin{prop}
  The ``arithmetic surface'' $X\times_\Funsq X$ is a topological space of dimension $2$. \qed
 \end{prop}

\begin{rem}
 The locally blueprinted space $\overz$ comes together with a morphism $\varphi\Spec\Z\to\overz$, which is a topological embedding. The image $Y$ of $\varphi$ (as a locally blueprinted space) is thus homeomorphic to $\Spec\Z$ and the underlying monoids of the structure sheaf are the same, only hte pre-additions of $\Spec\Z$ and $Y$ differ. In fact, $Y$ can be understood as a subspace of the blue scheme 
 \[ 
  V \quad = \quad \Spec\,\Bigl(\bpgenquot{(\Z,\cdot)}{1+(-1)\=0}\Bigr),
 \]
 which is an infinite-dimensional topological space. Then $Y$ consists of the generic point $\eta$ of $V$ and all its codimesnsion $1$ points, and the open sets of $Y$ stay in one-to-one correspondence to the open sets of $V$. The structure sheaf of $Y$ is isomorphic to the structure sheaf of $V$ as a functor from the topology of $Y$ resp. $V$ to blueprints. 
 
 Similarly, we find blue schemes underlying other affine (i.e.\ proper) open subsets of $X$. This might serve as a better explanation why we say that the structure sheaf of $X$ coherent.
\end{rem}

\section{Formulas for the Riemann zeta function}
\label{section: formulas for the riemann zeta function}

\noindent
The completed Riemann zeta function $\zeta^\ast(s)$ can be written as certain integrals over spaces that are connected to $\overz$.

\subsection{The ideal space}\label{subsection: ideal space}

Recall that $\cI_p$ denotes the set of all ideal of the stalk $\cO_{X,p}=\bpgenquot{A_p}{1+(-1)\=0}$ at $p$ for every prime $p\leq\infty$. We endow the sets $\cI_p$ with the following topologies and measures. For $p<\infty$, we define a basis of the topology as the sets
\[
 U_{b,c} \quad = \quad \{\ (p^l) \ | \ b\leq l \leq c \ \}
\]
where $b,c\in\N\cup\{\infty\}$ and $b\leq c$, with the usual convention that $(p^\infty)=(0)$. Then $\cI_p$ is nothing else than the one-point compactification of the discrete set of natural numbers. For every $s\in\C$, we define a measure $\mu_{p,s}$ on $\cI_p$ by
\[
 \mu_{p,s}(U_{b,c}) \quad = \quad \sum_{l=b}^c p^{-ls}.
\]
For $p=\infty$, we define a basis of the topology on $\cI_\infty$ as the sets
\[
 U_{b,c} \quad = \quad \{ \ (-a,a) \ | \ b\leq a \leq c \ \} \ \amalg \ \{ \ [-a,a] \ | \ b<a<c \ \} 
\]
for $0\leq b\leq c\leq 1$. Then the map $a\mapsto [-a,a]$ is a topological embedding of the interval $[0,1]$ into $\cI_\infty$. For every $s\in\C$, we define a measure $\mu_{\infty,s}$ on $\cI_\infty$ by
\[ 
 \mu_{\infty,s}(U_{b,c}) \quad = \quad \pi^{-s/2}\ \int\limits_b^c \ \bigl(-\ln x\bigr)^{s/2-1} \ dx.
\]
\begin{prop}
 For $p\leq\infty$ and $\Re(s)>1$, we have 
 \[ 
  \zeta_p(s) \quad = \quad \int\limits_{\cI_p} d\mu_{p,s}.
 \]
 In particular, the integral on the right hand side converges.
\end{prop}

\begin{proof}
 Let $\Re(s)>1$ and $p<\infty$. Then 
 \[
  \int\limits_{\cI_p} d\mu_{p,s} \quad = \quad \sum_{l=0}^\infty p^{-ls} \quad = \quad \zeta_p(s).
 \]
 Let $p=\infty$. 
Then the variable substitution $t=-\ln x$ yields
 \[
   \int\limits_{\cI_\infty} d\mu_{\infty,s} \quad = \quad \pi^{-s/2}\ \int\limits_0^1 \ \bigl(-\ln x\bigr)^{s/2-1} \ dx \quad = \quad \pi^{-s/2}\ \int\limits_0^\infty \ e^{-t}t^{s/2-1}\ dt \quad = \quad \pi^{-s/2}\Gamma\bigl(\frac s2\bigr) \quad = \quad \zeta_\infty(s). \qedhere
 \]
\end{proof}

If we define the \emph{ideal space} $\cI$ as the product space $\prod\cI_p$ over all primes $p\leq \infty$ and $\mu_s=(\mu_{p,s})$ as the product measure on $\cI$, then we obtain as an immediate consequence that
 \[ 
  \zeta^\ast(s) \quad = \quad \int\limits_{\cI} d\mu_{s}.
 \]

 The ideal space can also be understood as an adelic version of the space of ideals $\cI_\Z$ of the integers $\Z$, and in fact, $\cI$ can be endowed with the structure of a monoid. It seems to be interesting to work out the precise relationship of $\cI$ to the Bost-Connes system $\Q[\Q/\Z]\rtimes\cI_\Z$.

\subsection{Arakelov divisors}\label{subsection: arakelov divisors}

Let $M$ be a pointed set. We denote the base point of $M$ by $\ast_M$ or simply by $\ast$. A \emph{pre-addition on $M$} is an equivalence relation $\cP$ on the semigroup $\N[M]=\{\sum a_i|a_i\in M\}$ of finite formal sums in $M$ with the following properties (as usual, we write $\sum m_i\=\sum n_j$ if $\sum m_i$ stays in relation to $\sum n_j$):
\begin{enumerate}
 \item $\sum m_i\=\sum n_j$ and $\sum p_k\=\sum q_l$ implies $\sum m_i+\sum p_k\=\sum n_j+\sum q_l$,
 \item $\ast\=(\text{empty sum})$, and 
 \item if $m\=n$, then $m=n$ (in $M$).
\end{enumerate}
Let $B=\bpquot A\cR$ be a blueprint. A \emph{blue module on $B$} is a set $M$ together with a pre-addition $\cP$ and a \emph{$B$-action $B\times M\to M$}, which is a map $(b,m)\mapsto b.m$ that satisfies the following properties:
\begin{enumerate}
 \item $1.m=m$, $0.m=\ast$ and $a.\ast=\ast$,
 \item $(ab).m=a.(b.m)$, and
 \item $\sum a_i\=\sum b_j$ and $\sum m_k\=\sum n_l$ implies $\sum a_i.m_k\=\sum b_j.n_l$.
\end{enumerate}
A \emph{blue $\cO_X$-module $\cM$} is a sheaf on $X$ that associates to every open set $U\subset X$ a blue $\cO_X(U)$-module $\cM(U)$ satisfying the usual compatibility condition with respect to the restriction maps. A \emph{line bundle on $X$} is an $\cO_X$-module that is locally isomorphic to $\cO_X$. As in the case of line bundles on a curve over a field, the tensor product induces a group structure on the set $\Pic X$ of isomorphism classes of line bundles, which we call the \emph{Picard group of $X$}.

Given a line bundle $\cL$ on $X$, we can fix isomorphisms between its stalks $\cL_p$ and $\cO_{X,p}$ for every prime $p\leq\infty$, and between its generic stalks $\cL_\eta$ and $\cO_{X,\eta}$. The specialization maps yield embeddings
\[
 \varphi_p: \quad \bpgenquot{A_p}{1+(-1)\=0} \ = \ \cO_{X,p} \ \stackrel\sim\longrightarrow \ \cL_p \ \longrightarrow \ \cL_\eta \ \stackrel\sim\longrightarrow \ \cO_{X,\eta} \ = \ \bpgenquot{(\Q,\cdot)}{1+(-1)\=0}
\]
for every $p\leq\infty$. Define $n_p:=\norm{\varphi_p(1)}_p$ where $\norm\blanc _p$ is the usual $p$-adic absolute value of $\Q$. Note that only finitely many $n_p$ are non-zero. This defines an Arakelov divisor 
\[
 D(\cL) \quad = \quad \sum_{p\leq\infty}\, n_p\cdot p \qquad \in \quad \Div^\arak X \ = \ \bigoplus_{p<\infty}p^\Z\oplus\R_{>0}
\]
This divisor is well-defined up to the choice of the isomorphism $\cL_\eta\to\cO_{X,\eta}$ or, equivalently, up to the choice of a \emph{principal divisor}, which is an Arakelov divisor $D=\sum n_p\cdot p$ of norm $N(D)=\prod n_p=1$ (since the class number of $\Q$ is $1$). Let $P(X)$ be the subgroup of all principal divisors and $\Cl X=\Div X/P(X)$ the \emph{Arakelov divisor class group of $X$}. In its idelic topology, $\Cl X$ is a locally compact group. For a certain \emph{effectivity measure} $\mu_\eff$ on $\Cl X$, Van der Geer and Schroof establish in \cite{Geer-Schroof00} the formula\footnote{as pointed out to me by Bora Yalkinoglu}
\[
 \zeta^\ast(s) \quad = \quad \int\limits_{\Cl X} N(D)^{-s}\;d\mu_\eff,
\]
which is basically a reformulation of a Tate integral. 

The association $\cL\mapsto D(\cL)$ defines an inclusion
\[
 \Pic X \quad \hookrightarrow \quad \Cl X
\]
of groups, whose image is the set of divisor classes of the form $[\sum n_p\cdot p]$ with $n_\infty\in\Q_{>0}$. Note that $\Pic X$ is a dense subgroup of $\Cl X$ and the topologies of $\Cl X$ and $\Pic X$ generate the same $\sigma$-algebra. Therefore we might rewrite the above formula as 
\[
 \zeta^\ast(s) \quad = \quad \int\limits_{\Pic X} N(\cL)^{-s}\;d\mu_\eff.
\]

\section{\'Etale cohomology?}
\label{section: cohomology}

\noindent
In order to find a cohomological interpretation of the completed zeta function $\zeta^\ast(s)$, one might try to transfer the formalism of \'etale cohomology to the setting of this note. We spend these last paragraphs with a description of a possible definition of \'etale cohomology for $X=\overz$.

Since it is not clear what a ``separable field extension'' of the ``residue fields'' $\kappa(x)$ for $x\in X$ should be (note that $\kappa(x)$ is \emph{not} a field), it is not clear what an unramified morphism of locally blueprinted spaces should be. Therefore, it is better to work with the notion of formally \'etale morphisms of finite presentation, which makes sense in the context of locally blueprinted spaces. This approach is, for instance, the same as used in non-archimedean analytic geometry and in log-geometry.

The various equivalent definitions of sheaf cohomology give rise to different theories for locally blueprinted spaces. Sheaf cohomology can be defined in terms of injective resolutions. This gives rise to a well-working formalism (see \cite{Deitmar11b}). However, injective objects in the category of blue modules over a blueprint are very different in nature from injective objects in usual module categories. For example, the first cohomology group of the projective line $\P^1_\Fun$ is infinite dimensional (see \cite{L11}), which stays in complete contrast to cohomology of the projective line over a ring. One might still hope that this cohomology helps in view towards the Riemann hypothesis since the first cohomology group of $\overz$ is expected to infinite dimensional. But, morally speaking, The reason for the infinite dimensionality of $H^1(\overz,\cO_\cT)$ should not lie in the different behaviour of injective objects, but in the fact that $\Z$ is not finitely generated over $\Fun$ as a monoid.

Sheaf cohomology defined by extension groups gives also rise to infinite cohomology groups, which is due to the fact that there are infinite series of epimorphisms with the same kernel. In particular, epimorphisms of blue modules are, in general, not normal. If one restricts to normal epimorphisms and monomorphisms to define extension groups, then cohomology becomes trivial since extensions of blue modules are always a wedge product of the modules. Morally speaking, one can say that objects over $\Fun$ rigidify in the sense that coordinates are fixed.

The only hopeful approach to sheaf cohomology seems to be Cech cohomology. For a generalizaton of the definition of Cech cohomology, it is essential to be able to refer to the additive inverses of elements. In other words, we need the blue modules in question to have an action of $\Funsq$. Therefore, it is important to consider $\overz$ as a curve over $\Funsq$. In the toy example $\P^1_\Funsq$, we obtain indeed the expected sheaf cohomology for line bundles, i.e.\ if $h^i(\cL)$ denotes the dimension of $H^i(\overz,\cL)$ over $\Funsq$ and $\cO(n)$ is (the $\Funsq$-model of) the twisted structure sheaf of $\P^1_\Funsq$, then $h^0(\cO(n))=n+1$ for $n\geq0$ and $h^0(\cO(n))=0$ for $n<0$, and similarly for $h^1(\cO(n))$.

A common (but questioned) opinion is that the algebraic closure of $\Fun$ resp.\ $\Funsq$ is the union $\Funinf$ of all ``cyclotomic field extensions'' $\Funn$. The automorphism group of $\Funinf$ is $\widehat\Z^\times$ and its fixed sub\emph{field} is $\Funsq$. Thus $\widehat\Z^\times$ can be considered as the Galois group $\Gal(\Funinf/\Funsq)$ of $\Funinf$ over $\Funsq$, and the embedding of $\Funinf$ into the maximal abelian extension $\Q^{\textup{ab}}$ of $\Q$ induces an isomorphism $\Gal(\Funinf/\Funsq)\to\Gal(\Q^{\textup{ab}}/\Q)$.

The action of $G=\Gal(\Funinf/\Funsq)$ on $\Funinf$ defines an action on $\overX=\overz\times_\Funsq\Funinf$ by acting trivially on the points, but only acting arithmetically on the factor $\Funinf$ of the structure sheaf. This action might play the role of the absolute Frobenius in the number field case.

Up to this point, we see that all necessary definitions to set up \'etale cohomology of $\overz$ resp.\ $\overX$ have a straight forward generalization from usual scheme theory. Whether this will result in a valuable theory with respect to the Riemann hypothesis (or other arithmetic problems) is not clear to me at this moment, but certainly requires further investigation.


\begin{small}
\addcontentsline{toc}{section}{References}
 \bibliographystyle{plain}

\end{small}

\end{document}